\documentclass[a4paper]{article}

\usepackage[small,bf,hang]{caption} % Improved caption

\usepackage{amsmath,amsthm,amssymb}
\usepackage{graphicx,subcaption, tikz,color}
\usepackage{comment,xspace,paralist}							                              
\usepackage[shortlabels]{enumitem}	
\usepackage{hyperref}       
\hypersetup{
    colorlinks=true,
    citecolor = blue,
    linkcolor=black,
    filecolor=magenta,
    urlcolor=cyan,
    bookmarks=true,
}
\usepackage{times}
\usepackage{todonotes}
\usepackage[ruled,linesnumbered]{algorithm2e}
\newtheorem{theorem}{Theorem}
\newtheorem{lemma}{Lemma}
\newtheorem{claim}{Claim}

\title{{\bf  Some Results on $k$-Critical $P_5$-Free Graphs}}

\author{
Qingqiong Cai\thanks{College of Computer Science, Nankai University, Tianjin 300350, China.Supported by Natural Science Foundation of Tianjin (19JCQNJC14400).}
\and
Jan Goedgebeur \thanks{Department of Applied Mathematics, Computer Science and Statistics, Ghent University, 9000 Ghent, Belgium.} \thanks{Department of Computer Science, KU Leuven Campus Kulak, 8500 Kortrijk, Belgium.} 
\and
Shenwei Huang\thanks{The corresponding author. College of Computer Science, Nankai University, Tianjin 300350, China.  Supported by the National Natural Science Foundation of China (11801284) and Natural Science Foundation of Tianjin (20JCYBJC01190).}
}

\date{August 12, 2021}

\begin{document}

\maketitle

\begin{abstract} 
%Given two graphs $H_1$ and $H_2$, a graph $G$ is $(H_1,H_2)$-free if it contains no induced subgraph isomorphic to $H_1$ or $H_2$. 
%Let $P_t$ be the path on $t$ vertices.
A graph $G$ is $k$-vertex-critical if $G$ has chromatic number $k$ but every proper induced subgraph of $G$ has chromatic number less than $k$.
The study of $k$-vertex-critical graphs for graph classes is an important topic in algorithmic graph theory because
if the number of such graphs that are in a given hereditary graph class is finite, then there is a polynomial-time algorithm
to decide if a graph in the class is $(k-1)$-colorable. 

In this paper, we prove that for every fixed integer $k\ge 1$, there are only finitely many $k$-vertex-critical ($P_5$,gem)-free graphs and $(P_5,\overline{P_3+P_2})$-free graphs. To prove the results we use a known structure theorem for ($P_5$,gem)-free graphs combined with properties of $k$-vertex-critical graphs. 
 Moreover, we characterize all $k$-vertex-critical ($P_5$,gem)-free graphs
and $(P_5,\overline{P_3+P_2})$-free graphs for $k \in \{4,5\}$ using a computer generation algorithm.
\end{abstract}


%------------------------------------------------------------------------------------

\section{Introduction}

All graphs in this paper are finite and simple.
We say that a graph $G$ {\em contains} a graph $H$ if $H$ is
isomorphic to an induced subgraph of $G$.  A graph $G$ is
{\em $H$-free} if it does not contain $H$. 
For a family of graphs $\mathcal{H}$,
$G$ is {\em $\mathcal{H}$-free} if $G$ is $H$-free for every $H\in \mathcal{H}$.
When $\mathcal{H}$ consists of two graphs, we write
$(H_1,H_2)$-free instead of $\{H_1,H_2\}$-free.

A \emph{$q$-coloring} of a graph $G$ is a function $\phi:V(G)\longrightarrow \{ 1, \ldots ,q\}$ such that
$\phi(u)\neq \phi(v)$ whenever $u$ and $v$ are adjacent in $G$.
Equivalently, a $q$-coloring of $G$ is a partition of $V(G)$ into $q$ independent sets.
A graph is {\em $q$-colorable} if it admits a $q$-coloring.
The \emph{chromatic number} of a graph $G$, denoted by
$\chi (G)$, is the minimum number $q$ for which $G$ is $q$-colorable.
The \emph{clique number} of $G$, denoted by $\omega(G)$, is the size of a largest clique in $G$.

A graph $G$ is {\em $k$-chromatic} if $\chi(G)=k$. We say that $G$ is {\em $k$-critical} if it is
$k$-chromatic and $\chi(G-e)<\chi(G)$ for any edge $e\in E(G)$. For instance, $K_2$ is the only 2-critical graph
and odd cycles are the only 3-critical graphs. A graph is {\em critical} if it is $k$-critical for some integer $k\ge 1$.
Critical graphs were first defined and studied by Dirac~\cite{Di51,Di52,Di52i} in the early 1950s,
and then by Gallai and Ore~\cite{Ga63, Ga63i,Or67} among many others, and more recently by
Kostochka and Yancey~\cite{KY14}. 

A weaker notion of criticality is the so-called vertex-criticality.
A graph $G$ is {\em $k$-vertex-critical} if $\chi(G)=k$ and $\chi(G-v)<k$ for any $v\in V(G)$.
For a set $\mathcal{H}$ of graphs and a graph $G$, we say that $G$ is {\em $k$-vertex-critical $\mathcal{H}$-free}
if it is $k$-vertex-critical and $\mathcal{H}$-free.
We are mainly interested in the following question.

\noindent {\bf The meta question.} Given a set $\mathcal{H}$ of graphs and an integer $k\ge 1$, 
are there only finitely many $k$-vertex-critical  $\mathcal{H}$-free graphs?

This question is important in the study of algorithmic graph theory because of the following theorem, whose proof can be found in~\cite{CHLS19}.
\begin{theorem}[Folklore]\label{thm:finiteness}
Given a set $\mathcal{H}$ of graphs and an integer $k\ge 1$,
if the set of all $k$-vertex-critical $\mathcal{H}$-free graphs is finite, then there is a polynomial-time algorithm to determine
whether an $\mathcal{H}$-free graph is $(k-1)$-colorable.  
\end{theorem}

As usual, $P_t$ and $C_s$ denote
the path on $t$ vertices and the cycle on $s$ vertices, respectively. The complete
graph on $n$ vertices is denoted by $K_n$.
%The graph $K_3$ is also referred to as the {\em triangle}.
For two graphs $G$ and $H$, we use $G+H$ to denote the \emph{disjoint union} of $G$ and $H$.
For a positive integer $r$, we use $rG$ to denote the disjoint union of $r$ copies of $G$.
The \emph{complement} of $G$ is denoted by $\overline{G}$.
A {\em clique} (resp.\ {\em independent set}) in a graph is a set of pairwise adjacent (resp.\ nonadjacent) vertices.
If a graph $G$ can be partitioned into $k$ independent sets $S_1,\ldots,S_k$ such that there is an edge between every vertex in $S_i$
and every vertex in $S_j$ for all $1\le i<j\le k$, $G$ is called a {\em complete $k$-partite graph}; each $S_i$ is called a {\em part}
of $G$. If we do not specify the number of parts in $G$, we simply say that $G$ is a {\em complete multipartite graph}.
We denote by $K_{n_1,\ldots,n_k}$ the complete $k$-partite graph such that the $i$th part $S_i$ has size $n_i$, for each $1\le i\le k$.

In this paper, we continue to study $k$-vertex-critical graphs in the class of $P_5$-free graphs. 
Our research is mainly motivated by the following dichotomy result.

\begin{theorem}[\cite{CGHS21}]\label{thm:main}
Let $H$ be a graph of order 4 and $k\ge 5$ be a fixed integer.  
Then there are infinitely many $k$-vertex-critical $(P_5,H)$-free graphs if and only if
$H$ is $2P_2$ or $P_1+K_3$.
\end{theorem}

This Theorem completely settles the finiteness question of $k$-vertex-critical $(P_5,H)$-free graphs for graphs of order 4.
In~\cite{CGHS21}, the authors also posed the natural question of which five-vertex graphs $H$ lead to finitely many
$k$-vertex-critical $(P_5,H)$-free graphs. 
It is known that there are exactly 13 5-vertex-critical $(P_5,C_5)$-free graphs~\cite{HMRSV15},
and that there are finitely many 5-vertex-critical ($P_5$,banner)-free graphs~\cite{CHLS19, HLS19}, and
finitely many $k$-vertex-critical $(P_5,\overline{P_5})$-free graphs for every fixed $k$~\cite{DHHMMP17}.
Hell and Huang proved that there are finitely many $k$-vertex-critical $(P_6,C_4)$-free graphs~\cite{HH17}.
This was later generalized to $(P_t,K_{r,s})$-free graphs in the context of $H$-coloring~\cite{KP19}. 
This gives an affirmative answer for $H=K_{2,3}$.
Apart from these, there seem to be very few results on the finiteness of $k$-vertex-critical graphs for $k\ge 5$.
In this paper, we prove two new finiteness results beyond 5-vertex-criticality.

\subsection{Our Contributions} 

We continue to study the subclasses of $P_5$-free graphs.
In particular,  we focus on $(P_5,H)$-free graphs when $H$ has order 5.
The {\em gem} graph is the graph consisting of an induced $P_4$ with an additional vertex that is adjacent to every vertex on the $P_4$.
In this paper, we prove that there are only finitely many $k$-vertex-critical $(P_5,H)$-free graphs
for every fixed $k\ge 1$ when $H$ is gem or $\overline{P_3+P_2}$ (see \autoref{fig:two_graphs} for drawings of gem and $\overline{P_3+P_2}$). 
Moreover, we characterize all 4- and 5-vertex-critical ($P_5$,gem)-free graphs and $(P_5,\overline{P_3+P_2})$-free graphs by extending the computer generation algorithm from~\cite{GS18}.

To prove the result on gem-free graphs, we used a known structure theorem for ($P_5$,gem)-free graphs that is used to give an optimal
$\chi$-binding function for the class combined with an inductive argument based on a careful analysis of the structure. 
For the $\overline{P_3+P_2}$-free case, we performed a careful structural analysis combined with the repeated use of the pigeonhole principle based on the properties of $k$-vertex-critical graphs.

The remainder of the paper is organized as follows. We present some preliminaries
in \autoref{sec:pre} and prove our new results in \autoref{sec:new}. 
Finally, we give some concluding remarks in \autoref{sec:classification}.

\begin{figure}[tb]
\centering
\begin{subfigure}{.5\textwidth}
\centering
\begin{tikzpicture}[scale=0.6]
\tikzstyle{vertex}=[draw, circle, fill=white!100, minimum width=4pt,inner sep=1pt]
%\draw[step=1cm,color=gray] (-5,-5) grid (5,5);

\node[vertex] (1) at (2,0) {};
\node[vertex] (2) at (3,0) {};
\node[vertex] (3) at (4,0) {};
\node[vertex] (4) at (5,0) {};
\node[vertex] (5) at (3.5,2) {};

\draw (1)--(2)--(3)--(4);
\draw (5)--(1) (5)--(2) (5)--(3) (5)--(4);
\end{tikzpicture}
%\caption{Gem.}
\end{subfigure}%
\begin{subfigure}{.5\textwidth}
\centering
\begin{tikzpicture}[scale=0.6]
\tikzstyle{vertex}=[draw, circle, fill=white!100, minimum width=4pt,inner sep=1pt]
%\draw[step=1cm,color=gray] (-5,-5) grid (5,5);

\node[vertex] (1) at (0,1) {};
\node[vertex] (2) at (1,0) {};
\node[vertex] (3) at (0,-1) {};
\node[vertex] (4) at (-1,0) {};
\node[vertex] (5) at (0,-2) {};

\draw (1)--(2)--(3)--(4)--(1);
\draw (1)--(3) (5)--(2) (5)--(4);
\end{tikzpicture}
%\caption{$\overline{P_3+P_2}$.}
\end{subfigure}%
\caption{The graphs gem (left) and $\overline{P_3+P_2}$ (right).}\label{fig:two_graphs} %All of these graphs are also 4-critical ($P_5$,gem)-free.}
\end{figure}
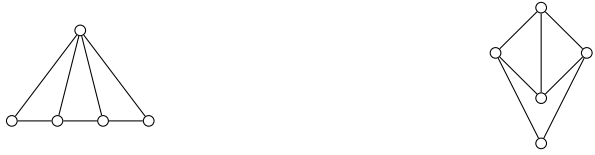

%------------------------------------------------------------------------------------

\section{Preliminaries}\label{sec:pre}

For general graph theory notation we follow~\cite{BM08}.
Let $G=(V,E)$ be a graph.  If $uv\in E$, we say that $u$ and $v$ are {\em neighbors} or {\em adjacent}; otherwise
$u$ and $v$ are {\em nonneighbors} or {\em nonadjacent}.
The \emph{neighborhood} of a vertex $v$, denoted by $N_G(v)$, is the set of neighbors of $v$.
For a set $X\subseteq V(G)$, let $N_G(X)=\bigcup_{v\in X}N_G(v)\setminus X$.
We shall omit the subscript whenever the context is clear.
For $X,Y\subseteq V$, we say that $X$ is \emph{complete} (resp.\ \emph{anticomplete}) to $Y$
if every vertex in $X$ is adjacent (resp.\ nonadjacent) to every vertex in $Y$.
If $X=\{x\}$, we write ``$x$ is complete (resp.\ anticomplete) to $Y$'' instead of ``$\{x\}$ is complete (resp.\ anticomplete) to $Y$''.
If a vertex $v$ is neither complete nor anticomplete to a set $S$, we say that $v$ is {\em mixed} on $S$.
We say that $H$ is a {\em homogeneous set} if no vertex in $V-H$ is mixed on $H$.
More generally, we say that $H$ is {\em homogeneous with respect to a subset} $S\subseteq V$ if no vertex in $S$ can be mixed on $H$.
A vertex is \emph{universal} in $G$ if it is adjacent to all other vertices.
Two nonadjacent vertices $u$ and $v$ are said to be \emph{comparable} if $N(v)\subseteq N(u)$ or $N(u)\subseteq N(v)$.
A vertex subset $K\subseteq  V$ is a \emph{clique cutset} if $G-K$ has more components than $G$ and $K$
induces a clique.  For $S\subseteq V$, the subgraph \emph{induced} by $S$, is denoted by $G[S]$.
A {\em $k$-hole} in a graph is an induced cycle $H$ of length $k\ge 4$. If $k$ is odd, we say that $H$ is an {\em odd hole}.
A {\em $k$-antihole} in $G$ is a $k$-hole in $\overline{G}$. Odd antiholes are defined analogously.
%The graph obtained from $C_k$ by adding a universal vertex, denoted by $W_k$, is called the {\em $k$-wheel}.

\medskip
We proceed with a few useful results that will be needed later. The first one is a folklore property of $k$-vertex-critical graphs.

\begin{lemma}[Folklore]\label{lem:clique cutsets}
A $k$-vertex-critical graph cannot contain clique cutsets.
\end{lemma}

Another folklore property of vertex-critical graphs is that such graphs contain no comparable vertices. A generalization of this property was presented in~\cite{CGHS21}.

\begin{lemma}[\cite{CGHS21}]\label{lem:dominated subsets}
Let $G$ be a $k$-vertex-critical graph. Then $G$ has no two nonempty disjoint subsets $X$ and $Y$ of $V(G)$ that satisfy all the following conditions.

\begin{itemize}
\item $X$ and $Y$ are anticomplete to each other.
\item $\chi(G[X])\le \chi(G[Y])$.
\item $Y$ is complete to $N(X)$.
\end{itemize}
\end{lemma}

%\begin{proof}
%Suppose that $G$ has a pair of nonempty subsets $X$ and $Y$ that satisfy all three conditions.
%Since $G$ is $k$-vertex-critical, $G-X$ has a $(k-1)$-coloring $\phi$.  Let $t=\chi(G[Y])$.
%Since $Y$ is complete to $N(X)$, at least $t$ colors do not appear on any vertex in $N(X)$ under $\phi$.
%So we can obtain a $(k-1)$-coloring of $G$ by coloring $G[X]$ with those $t$ colors. This contradicts
%that $G$ is $k$-chromatic.
%\end{proof}

A graph $G$ is {\em perfect} if $\chi(H)=\omega(H)$ for each induced subgraph $H$ of $G$.
An {\em imperfect} graph is a graph that is not perfect.
We conclude this section with the celebrated Strong Perfect Graph Theorem~\cite{CRST06}.

\begin{theorem}[Strong Perfect Graph Theorem~\cite{CRST06}]\label{thm:SPGT}
A graph is perfect if and only if it contains no odd holes or odd antiholes.
\end{theorem}

%------------------------------------------------------------------------------------

\section{New Results}\label{sec:new}

\subsection{Gem-Free Graphs}

\begin{theorem}\label{thm:gem}
For every fixed integer $k\ge 1$, there are finitely many $k$-vertex-critical ($P_5$, gem)-free graphs.
\end{theorem}

To prove \autoref{thm:gem}, we need a structure theorem for the class of ($P_5$, gem)-free graphs that was obtained in
\cite{CKMM20}. Let $\mathcal{H}$ be the class of connected ($P_5$, gem)-free graphs $G$ such that $V(G)$ can be partitioned
into seven nonempty sets $A_1,A_2,\ldots,A_7$ such that:

$\bullet$ Each $A_i$ induces a $P_4$-free graph.

$\bullet$ $A_1$ is complete to $A_2\cup A_5\cup A_6$ and anticomplete to $A_3\cup A_4\cup A_7$.

$\bullet$ $A_3$ is complete to $A_2\cup A_4\cup A_6$ and anticomplete to $A_5\cup A_7$.

$\bullet$ $A_4$ is complete to $A_5\cup A_6$ and anticomplete to $A_2\cup A_7$.
 
$\bullet$ $A_2$ is anticomplete to $A_5\cup A_6\cup A_7$.

$\bullet$ $A_5$ is anticomplete to $A_6\cup A_7$.

$\bullet$ The vertex set  of each component of $G[A_7]$ is a homogeneous set of $G$.

$\bullet$ The adjacency between $A_6$ and $A_7$ is not specified but it is restricted by the fact that $G$ is ($P_5$, gem)-free.

See \autoref{fig:calH} for an illustration.
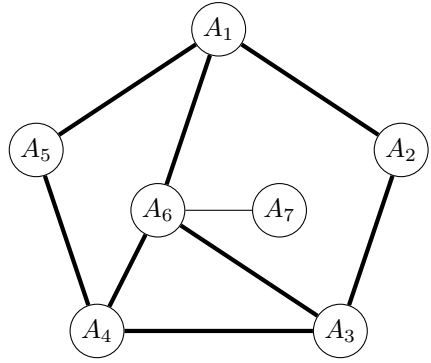
\begin{figure}[t]
\centering

\begin{tikzpicture}[scale=0.8]
\tikzstyle{vertex}=[draw, circle, fill=white!100, minimum width=4pt,inner sep=2pt]
%\draw[step=1cm,color=gray] (-6,-6) grid (6,6);

\node[vertex] (A1) at (0,2) {$A_1$};
\node[vertex] (A2) at (3,0) {$A_2$};
\node[vertex] (A5) at (-3,0) {$A_5$};
\node[vertex] (A3) at (2,-3) {$A_3$};
\node[vertex] (A4) at (-2,-3) {$A_4$};

\draw[ultra thick] (A1)--(A2)--(A3)--(A4)--(A5)--(A1);

\node[vertex] (A6) at (-1,-1) {$A_6$};
\node[vertex] (A7) at (1,-1) {$A_7$};
\draw[ultra thick] (A6)--(A1) (A6)--(A3) (A6)--(A4);
\draw (A6)--(A7);

\end{tikzpicture}

\caption{A diagram for graphs $G\in \mathcal{H}$. A thick line between two sets means that the two sets are complete.
A normal line between two sets means that the edges are arbitrary (subject to the ($P_5$,gem)-freeness). No line between two sets means that the two sets are anticomplete. Moreover, each $A_i$ is $P_4$-free, and the vertex set of each component of $G[A_7]$ is a homogeneous set of $G$.}\label{fig:calH}
\end{figure}

For two graphs $G$ and $H$, we say that $G$ is an {\em substitution} of $H$ if $V(G)$ can be partitioned into $|V(H)|$ nonempty
subsets $Q_v$, $v\in V(H)$ so that $Q_u$ and $Q_v$ are complete if $uv\in E(H)$ and anticomplete if $uv\notin E(H)$ for every
different vertices $u,v\in V(H)$. If each $Q_v$ is $P_4$-free, then it is called an {\em $P_4$-free substitution};
if each $Q_v$ is a clique, then it is called a {\em clique substitution}.
To state the structural theorem, we also need 10 special graphs $G_1,G_2,\ldots,G_{10}$ each of which has order at most 9.
We omit the drawing of these 10 graphs and refer the readers to~\cite{CKMM20}. Now we are ready to state the structure theorem
obtained in~\cite{CKMM20}.

\begin{lemma}[\cite{CKMM20}]\label{lem:p5gem}
Let $G$ be a connected ($P_5$, gem)-free graph that contains an induced $C_5$. Then either $G\in \mathcal{H}$ or $G$
is a $P_4$-free substitution of $G_1,G_2,\ldots,G_{10}$.
\end{lemma}

%Note that $G_1,G_2,\ldots,G_{10}$ each has order at most 9 and later we only need this fact. So we omit the drawing of these 10 graphs and refer the readers to~\cite{CKMM20}.

We also need a lemma from~\cite{CHM19}.

\begin{lemma}[\cite{CHM19}]\label{lem:clique homogeneous set}
Let $G$ be a connected ($P_5$, gem)-free graph and $H$ a homogeneous set of $G$ that is not a clique. 
Then there exists a connected induced subgraph $G'$ of G with
$|G'|<|G|$ such that $\omega(G')=\omega(G)$ and $\chi(G')=\chi(G)$.
\end{lemma}

We are now ready to prove the main result in this subsection.

\begin{proof}[Proof of \autoref{thm:gem}]
Let $G$ be a $k$-vertex-critical ($P_5$,gem)-free graph. Then $G$ is connected. 
We show that $|G|$ is bounded by a function of $k$.
If $G$ contains a $K_k$, then $G=K_k$. So we may assume that $G$ is $K_k$-free.
Since $G$ is not perfect, it follows from \autoref{thm:SPGT} that $G$ contains an induced $C_5$.
By \autoref{lem:p5gem}, $G\in \mathcal{H}$ or $G$ is a $P_4$-expansion of $G_1,G_2,\ldots,G_{10}$.
If $G$ is a $P_4$-free substitution of $G_1,G_2,\ldots,G_{10}$, then it must be a clique expansion of 
$G_1,G_2,\ldots,G_{10}$ by \autoref{lem:clique homogeneous set}. Since $|G_i|\le 9$ for $1\le i\le 10$ (see~\cite{CKMM20}),
it follows that $|G|\le 9k$. 

Therefore, we assume that $G\in \mathcal{H}$ in the following and can be partitioned into seven nonempty
subsets $A_1,A_2,\ldots,A_7$ as described in the definition. Note that each $A_i$ ($1\le i\le 5$) is a homogeneous set of $G$.
Thus, $A_i$ (for $1\le i\le 5$) is a clique by \autoref{lem:clique homogeneous set} and so $|A_i|\le k$. 
So it remains to bound $A_6$ and $A_7$. For the same reason, every component of $G[A_7]$ is a clique and has size
at most $k$.

\begin{claim}\label{clm:induction}
Let $X\subseteq A_6$ be a subset such that $X$ is homogeneous with respect to $V(G)-A_7$. Then 
$|X|\le 2^{2\omega^2(X)}$.
\end{claim}

\begin{proof}
We prove this by induction on $\omega(X)$. Since $A_3\neq \emptyset$, it contains a vertex, say $t$.
The base case is $\omega(X)=1$, i.e., $X$ is an independent set.
If $X$ contains two vertices $x_1,x_2$, then by \autoref{lem:dominated subsets} and our assumption
there exist vertices $y_1,y_2\in A_7$
such that $x_1y_1,x_2y_2\in E$ but $x_1y_2,x_2y_1\notin E$. Since every component of $G[A_7]$ is a homogeneous set of $G$,
$y_1y_2\notin E$. Then $y_1,x_1,t,x_2,y_2$ induces a $P_5$.  This proves that $|X|\le 1$.

Suppose now that $\omega:=\omega(X)\ge 2$ and the claim is true for any subset $Y$ such that $Y$ is homogeneous with respect to $V(G)-A_7$ with $\omega(Y)<\omega$. We first prove two useful facts about components of $X$.

\begin{equation} \label{eqn_homogeneous components}
\text{All but at most one component of $X$ are homogeneous sets of $G$.}
\end{equation}

\noindent {\em Proof of (\ref{eqn_homogeneous components}).}
Suppose not. Then there exist two components $C$ and $C'$ of $X$ such that $C$ and $C'$ are not homogeneous sets of $G$.
By our assumption on $X$, there are vertices $v,v'\in A_7$ such that $v$ is mixed on an edge $ab\in C$ with $va\notin E$ and $vb\in E$, and $v'$ is mixed on an edge $a'b'\in C'$ with $v'a'\notin E$ and $v'b'\in E$.

If $v,v'$ are in the same component of $A_7$, then $vv'\in E$. Since every component of $A_7$ is a homogeneous set of $G$,
$vb',v'b\in E$ and $va',v'a\notin E$. Then $a,b,v,b',a'$ induces a $P_5$. 

So $v,v'$ lie in different components of $A_7$. If $vb',v'b\notin E$, then $v,b,t,b',v'$ induces a $P_5$. 
So by symmetry we assume that $vb'\in E$. Then $va'\in E$ since otherwise $a,b,v,b',a'$ induces a $P_5$.
Note that $a,b,v,b',v'$ induces a $P_5$ unless $v'$ is adjacent to $a$ or $b$. If $v'b\in E$, then $v'a\in E$
since otherwise $a,b,v',b',a'$ induces a $P_5$. This shows that $v'a\in E$.
But now $va'$ and $v'a$ induce a $2P_2$ and this together with $t\in A_3$ gives an induced $P_5$.
This completes the proof of (\ref{eqn_homogeneous components}).

\begin{equation} \label{eqn_2}
\text{There is at most one homogeneous component of $X$ of size $i$ for $1\le i\le \omega$.}
\end{equation}

\noindent {\em Proof of (\ref{eqn_2}).} 
Suppose not. Then there are two homogeneous components $C$ and $C'$ of $X$ of size $i$ for some $1\le i\le k$.
By \autoref{lem:clique homogeneous set}, we know that $C$ and $C'$ are two cliques of size $i$, and thus $\chi(G[C])=\chi(G[C'])=i$.
By \autoref{lem:dominated subsets}, there exist $a,a'\in A_7$ such that $a\in N(C)\setminus N(C')$ and 
$a'\in N(C')\setminus N(C)$. Then $a,c,t,c',a'$ induces a $P_5$, where $c\in C$, $c'\in C'$.
This completes the proof of (\ref{eqn_2}).

\medskip 
Let $H$ be the possible non-homogeneous component of $X$. Then $|X|\le \omega^2+|H|$. We now bound $H$.
Let $K=\{v_1,v_2,\ldots,v_s\}$ be a maximum clique of $H$. Clearly, $s\le \omega(X)$.
For every $I\subseteq K$, define
\[
	T_I=\{v\in V(H)\setminus V(K):N_K(v)=I\}.
\]
We show that $T_I$ is homogeneous with respect to $V(G)-A_7$. First note that $T_{K}=\emptyset$ by the choice of $K$.
Suppose that $T_{\emptyset}\neq \emptyset$. Then by the connectivity of $H$, there is a vertex $z\in T_{\emptyset}$
that has a neighbor $y\in T_I$ for some $I\neq \emptyset$. Since $y$ cannot be complete to $K$, 
there exist $i,j$ such that $y$ is adjacent to $v_i$ but not to $v_j$. Then $z,y,v_i,v_j$ induces a $P_4$, which contradicts
the fact that $A_6$ is $P_4$-free. So $T_{\emptyset}= \emptyset$.

Let $I,J\subseteq K$ with $I\neq J$. If there exist $i,j$ such that $v_i\in I\setminus J$ and $v_j\in J\setminus I$,
then $T_I$ and $T_J$ are complete: if $t_i\in T_I$ and $t_j\in T_j$ are not adjacent, then $t_i,v_i,v_j,t_j$ induces a $P_4$.
Now suppose that $I\subseteq J$. Since $|I|<|J|$, there exists an index $j$ such that $v_j\in J\setminus I$. 
Since $J\neq K$, there is an index $k$ such that $v_k\notin J$. Then if $t_i\in T_I$ and $t_j\in T_J$ are adjacent,
then $t_i,t_j,v_j,v_k$ induces a $P_4$. So $T_I$ and $T_J$ are anticomplete.

This proves that for each $I\subseteq K$, $T_I$ is homogeneous with respect to $V(G)-A_7$. 
Moreover, $\omega(T_I)<\omega$. By the inductive hypothesis, it follows that $|T_I|\le 2^{2(\omega-1)^2}$. 
So
$$|H|=\sum_{I\subseteq K}|T_I|+|K|\le 2^{2(\omega-1)^2}2^{\omega}+\omega.$$

Thus, $|X|\le 2^{2(\omega-1)^2}2^{\omega}+\omega+\omega^2=2^{2\omega^2-3\omega+2}+(\omega+\omega^2)$.
Since $\omega+\omega^2\le 2^{3\omega-2}$, it follows that $|X|\le 2^{2\omega^2}$.
%Sovling the recurrence relation $f(t)\le f(t-1)2^t+t+t^2$, we have $f(t)\le 2^{O(t^2)}$.
This proves the claim.
\end{proof}

Applying \autoref{clm:induction} to $A_6$, we have that $|A_6|\le 2^{2\omega^2(A_6)}\le 2^{2k^2}$. Finally, we bound $A_7$.

\begin{claim}
$|A_7|\le k2^{|A_6|}$.
\end{claim}

\begin{proof}
Recall that each component of $A_7$ is a clique and so has size at most $k$.
We now group the components by size and show that each group has at most $2^{|A_6|}$ components.
Suppose not. Then there are two components $C$ and $C'$ of $A_7$ having the same size $i$ with
$N(C)=N(C')$, which contradicts \autoref{lem:dominated subsets}. This proves the claim. 
\end{proof}

Therefore, $|G|=\sum_{i=1}^7|A_i|\le 5k+|A_6|+|A_7|$, which is a function of $k$.
\end{proof}

\subsection{$\overline{P_3+P_2}$-Free Graphs}

\begin{theorem}\label{thm:diamond}
For every fixed integer $k\ge 1$, there are finitely many $k$-vertex-critical ($P_5$, $\overline{P_3+P_2}$)-free graphs.
\end{theorem}

\begin{proof}
Let $G$ be a $k$-vertex-critical ($P_5$, $\overline{P_3+P_2}$)-free graph. We show that $|G|$ is bounded by a function of $k$.
If $G$ contains a $K_k$, then $G$ is isomorphic to $K_k$ and thus $|G|=k$.
So assume that $G$ is $K_k$-free. Since $G$ is imperfect, $G$ contains an induced $C_5$ by \autoref{thm:SPGT}.
Let $Q=v_1,v_2,v_3,v_4,v_5$ be an induced $C_5$. For each $1\le i\le 5$, we define
\begin{equation*} \label{eq1}
\begin{split}
Z & = \{v\in V\setminus Q: N_{Q}(v)=\emptyset\}, \\
C_i  & = \{v\in V\setminus Q: N_{Q}(v)=\{v_{i-1},v_{i+1}\}\}, \\
Y_i  & = \{v\in V\setminus Q: N_{Q}(v)=\{v_{i-2}, v_i, v_{i+2}\}\}, \\
T_i  & = \{v\in V\setminus Q: N_{Q}(v)=\{v_{i-1}, v_i, v_{i+1}\}\}, \\
F_i  & = \{v\in V\setminus Q: N_{Q}(v)=V(Q)\setminus \{v_i\}\}, \\
U & = \{v\in V\setminus Q: N_{Q}(v)=V(Q)\}.\\
\end{split}
\end{equation*}
Let $C=\cup_{1\le i\le 5}C_i$, $Y=\cup_{1\le i\le 5}Y_i$, $T=\cup_{1\le i\le 5}T_i$ and $F=\cup_{1\le i\le 5}F_i$. 
We choose $Q$ so that $|Y|$ is maximum.
%It follows immediately from the $P_5$-freeness of $G$ that $V(G)=V(Q)\cup Z\cup C\cup Y\cup T\cup F\cup U$.

\begin{claim}\label{clm:partition}
$V(G)=V(Q)\cup Z\cup C\cup Y\cup T\cup F\cup U$.
\end{claim}

\begin{proof}
Suppose that the claim is false. Then there is a vertex $u$ such that $\{v_i\}\subseteq N(u)\subseteq \{v_i,v_{i+1}\}$ for some 
$1\le i\le 5$. Then $u,v_i,v_{i+4},v_{i+3},v_{i+2}$ induces a $P_5$.
\end{proof}

We first bound $U$.

\begin{claim}\label{clm:U}
For each $1\le i\le 5$, $Y_i\cup U$ and $F_i\cup U$ are cliques.
\end{claim}

\begin{proof}
Suppose that $Y_i \cup U$ contains two nonadjacent vertices $x$ and $y$, then $\{x,y,v_{i-2},v_i,v_{i+2}\}$ induces a $\overline{P_3+P_2}$. The proof for $F_i\cup U$ is completely analogous.
\end{proof}

By \autoref{clm:U}, it follows that $|U|<k$.

\begin{claim}\label{clm:Yi}
For each $1\le i\le 5$, $C_i$, $Y_i$ and $F_i$ are independent sets.
\end{claim}

\begin{proof}
Note that for each vertex $x\in C\cup Y\cup F$, there is an index $i$ such that $x$ is not adjacent to $v_i$ but adjacent to
both $v_{i+1}$ and $v_{i-1}$. If $C_i$ contains two adjacent vertices $x$ and $y$, then let $j$ be such that
$x,y$ are not adjacent to $v_j$ but adjacent to both $v_{j-1}$ and $v_{j+1}$. Then $\{x,y,v_{j-1},v_j,v_{j+1}\}$ induces a $\overline{P_3+P_2}$.
The proof for $Y_i$ and $F_i$ is completely analogous.
\end{proof}

By \autoref{clm:U} and \autoref{clm:Yi}, $|Y_i|\le 1$ and $|F_i|\le 1$. So $|Y|\le 5$ and $|F|\le 5$.
Next we bound $Z$.

\begin{claim}\label{clm:ZTC}
$Z$ is anticomplete to $C\cup T$.
\end{claim}

\begin{proof}
Let $z\in Z$. If $z$ has a neighbor $x\in C_i\cup T_i$, then $z,x,v_{i+1},v_{i+2},v_{i+3}$ induces a $P_5$.
\end{proof}

\begin{claim}\label{clm:ZY}
Each vertex in $Y\cup F$ is either complete or anticomplete to each component of $Z$.
\end{claim}

\begin{proof}
Let $v\in Y\cup F$ be that $v$ is adjacent to $a$ but not adjacent to $b$ with $ab\in E(Z)$. 
Note that there is an index $i$ such that $v,v_i,v_{i+1}$ is an induced $P_3$.
Then $b,a,v,v_i,v_{i+1}$ induces a $P_5$. 
\end{proof}

\begin{claim}
$Z$ is an independent set.
\end{claim}

\begin{proof}
Suppose not. Let $K$ be a component of $Z$ containing an edge $ab$. By \autoref{clm:ZTC}, $N(K)\subseteq Y\cup F\cup U$.
Since $G$ has no clique cutset, it follows from \autoref{clm:ZY} and \autoref{clm:U} that there are two nonadjacent
vertices $x$ and $y$ in $Y\cup F$ that are complete to $K$. Let $v_i\in V(Q)$ be a common neighbor of $x$ and $y$.
Then $\{a,b,x,y,v_i\}$ induces a $\overline{P_3+P_2}$.
\end{proof}

By \autoref{clm:ZTC}, $N(z)\subseteq Y\cup F\cup U$ for each $z\in Z$. Note that $|Y\cup F\cup U|\le k+10$ and so
there are $2^{10+k}$ possible neighborhoods for vertices in $Z$.
By \autoref{lem:dominated subsets} and the pigeonhole principle, $|Z|\le 2^{10+k}$.

\medskip 
Next we bound $C$.

\begin{claim}\label{clm:CiCi+1}
$C_i$ and $C_{i+1}$ are complete to each other.
\end{claim}

\begin{proof}
Let $c_3\in C_3$ and $c_4\in C_4$. If $c_3c_4\notin E$, then $c_4,v_5,v_1,v_2,c_3$ induces a $P_5$.
\end{proof}

\begin{claim}\label{clm:CiTi}
$C_i$ is complete to $T_i$.
\end{claim}

\begin{proof}
By symmetry, we prove for $i=1$. If $c_1\in C_1$ and $t_1\in T_1$ are not adjacent, then $\{c_1,t_1,v_5,v_1,v_2\}$ induces a $\overline{P_3+P_2}$.
\end{proof}

\begin{claim}\label{clm:CiTi+1}
$C_i$ is anticomplete to $T_{i+1}$ and $T_{i-1}$.
\end{claim}

\begin{proof}
If $c_1\in C_1$ is adjacent to $t_2\in T_2$, then $\{v_5,c_1,v_1,v_2,t_2\}$ induces a $\overline{P_3+P_2}$. 
The case for $T_5$ is similar.
\end{proof}

\begin{claim}\label{clm:CiTi+2}
No vertex in $T_{i+2}$ or $T_{i-2}$ can be mixed on $C_i$.
\end{claim}

\begin{proof}
Suppose that $t_3\in T_3$ is mixed on $C_1$. Then there are vertices $c_1,c'_1\in C_1$ with $t_3c_1\in E$ and $t_3c'_1\notin E$. 
Then $c'_1,v_5,c_1,t_3,v_3$ induces a $P_5$.
The case for $T_4$ is similar.
\end{proof}

\begin{claim}\label{clm:boundCi}
For each $1\le i\le 5$, $|C_i|\le 2^{|Y\cup F\cup U|}$.
\end{claim}

\begin{proof}
Suppose not. Then by the pigeonhole principle, there are two vertices $c_1,c_2\in C_1$ such that $c_1$ and $c_2$ have the same neighbors in $Y\cup F\cup U$. By \autoref{clm:ZTC} and \autoref{clm:CiCi+1}-\autoref{clm:CiTi+2}, it follows that $c_1$ and $c_2$ have the
same neighbors in $Z\cup T\cup C_1\cup C_2\cup C_5$. Since $c_1$ and $c_2$ are not comparable, there exist $d_1\in N(c_1)\setminus N(c_2)$
and $d_2\in N(c_2)\setminus N(c_1)$. So $d_1,d_2\in C_3\cup C_4$.
If $d_1,d_2\in C_3$, then $d_1,c_1,v_5,c_2,d_2$ induces a $P_5$.
So we may assume that $d_1\in C_4$ and $d_2\in C_3$.

Consider the induced five-cycle $Q'=Q-\{v_1\}\cup \{c_1\}$.
Let $Y'$ be the set of vertices in $V(G)\setminus V(Q')$ whose neighborhood on $Q'$ induces a $P_2+P_1$. 
Note that $d_1\in Y'\setminus Y$. By the choice of $Q$, there exists a vertex $y\in Y\setminus Y'$.
Clearly, $y\in Y_1\cup Y_3\cup Y_4$ and $yc_1\notin E$. If $y\in Y_1$, then $c_1,v_5,v_1,y,v_3$ induces a $P_5$.
So $y\in Y_3\cup Y_4$. By symmetry, we may assume that $y\in Y_4$. 
Note that $yc_2\notin E$ since otherwise $\{y,c_2,v_1,v_2,v_5\}$ induces a $\overline{P_3+P_2}$, $yd_1\in E$ since otherwise $y,v_1,v_5,d_1,v_3$ induces a $P_5$, and that
$yd_2\notin E$ since otherwise $\{y,d_1,d_2,v_2,v_3\}$ induces a $\overline{P_3+P_2}$.
But now $c_2,d_2,d_1,y,v_1$ induces a $P_5$, a contradiction.
\end{proof}

By \autoref{clm:boundCi}, $|C|\le 5\times 2^{k+10}$. Finally, we bound $T$.

\begin{claim}\label{clm:Ti}
For each $1\le i\le 5$, $T_i$ is $P_2+P_1$-free and thus induces a complete multipartite graph.
\end{claim}

\begin{proof}
If $T_i$ contains three vertices $x,y,z$ such that $xy$ and $z$ induce a disjoint union of an edge and a vertex,
then $\{x,y,z,v_{i-1},v_{i+1}\}$ induces a $\overline{P_3+P_2}$. It is well-known that a graph is a
complete multipartite graph if and only if it is $P_2+P_1$-free.
\end{proof}

\begin{claim}\label{clm:TiTi+1}
For each $1\le i\le 5$, $T_i$ and $T_{i+1}$ are complete.
\end{claim}

\begin{proof}
If $t_1\in T_1$ and $t_2\in T_2$ are not adjacent, then $t_2,v_3,v_4,v_5,t_1$ induces a $P_5$.
\end{proof}

\begin{claim}\label{clm:boundT}
$|T_i|\le k2^{|Z\cup C\cup Y\cup F\cup U|}$.
\end{claim}

\begin{proof}
By \autoref{clm:Ti}, $T_1$ is a complete $r$-partite graph $(X_1,\ldots,X_r)$ for some $r\le k$ where $X_i$ and $X_j$ are complete for all $1\le i<j\le r$. If $|X_i|>2^{|Z\cup C\cup Y\cup F\cup U|}$, then by the pigeonhole principle there are two nonadjacent vertices $x_1$ and $x_2$ that have the same neighbors in  $Z\cup C\cup Y\cup F\cup U$. Since $x_1$ and $x_2$ are not comparable, there are vertices $y_1\in N(x_1)\setminus N(x_2)$ and $y_2\in N(x_2)\setminus N(x_1)$. By \autoref{clm:TiTi+1}, $y_1,y_2\in T_3\cup T_4$. Then $v_4,y_1,x_1,v_1,x_2$
induces a $P_5$.
\end{proof}

By \autoref{clm:boundT}, $|T|$ is bounded by a function of $k$ and so is $|G|$. This completes the proof.
\end{proof}

\subsection{Complete Characterization for $k=4$ and $k=5$}

We remark that the upper bound on the order of the $k$-vertex-critical graphs in \autoref{thm:diamond} and  \autoref{thm:gem} 
is far from being tight: the objective was to show the {\em existence} of such a bound and so we did not try to optimize the bound. This indicates that it may be possible to characterize all $k$-vertex-critical graphs for small values of $k$. In this subsection, we prove that this is indeed possible for $k=4$ and $k=5$. In particular, we prove the following theorems.

\begin{theorem}\label{thm:k=4}
There are exactly 3 4-vertex-critical ($P_5$,gem)-free graphs and there are exactly 6 4-vertex-critical $(P_5,\overline{P_3+P_2})$-free graphs.
\end{theorem}

\begin{theorem}\label{thm:k=5}
There are exactly 7 5-vertex-critical ($P_5$,gem)-free graphs and there are exactly 20 5-vertex-critical $(P_5,\overline{P_3+P_2})$-free graphs.
\end{theorem}

\begin{proof}[Proof of \autoref{thm:k=4} and \autoref{thm:k=5}]
The proof uses the computer program for generating $k$-vertex-critical $\mathcal{H}$-free graphs developed in~\cite{GS18}. The program takes a positive
integer $k$ and a set $\mathcal{H}$ of forbidden induced subgraphs as input and if it terminates, it outputs a complete list of 
$k$-vertex-critical $\mathcal{H}$-free graphs. The source code of this algorithm can be downloaded from~\cite{criticalpfree-site}.  We refer to~\cite{GS18} for more details about the generation algorithm.
The correctness of this algorithm has extensively been tested in the literature, see~\cite{CGHS21,CGSZ20, GS18} for example.

We extended the algorithm from~\cite{GS18} so it can also generate $H$-free graphs where $H= $ gem or $H=\overline{P_3+P_2}$ and executed the algorithm in with the following inputs:
\begin{itemize}
\item $k=4$ or $k=5$;
\item $\mathcal{H}=\{P_5$,gem$\}$ or $\mathcal{H}=\{P_5,\overline{P_3+P_2}\}$.
\end{itemize}

In each of these 4 cases the algorithm terminates within a few seconds and yields the graphs reported in \autoref{thm:k=4}, \autoref{thm:k=5}, and in \autoref{table:counts_vertex_critical_graphs}.
\end{proof}

The counts of the vertex-critical graphs from \autoref{thm:k=4} and \autoref{thm:k=5} are listed in \autoref{table:counts_vertex_critical_graphs}.
The 3 4-vertex-critical ($P_5$,gem)-free graphs from \autoref{thm:k=4} are shown in \autoref{fig:4vcP5gem} and the 7 5-vertex-critical ($P_5$,gem)-free graphs from \autoref{thm:k=5} are shown in \autoref{fig:5vcP5gem}.

The adjacency lists of all vertex-critical graphs from \autoref{thm:k=4} and \autoref{thm:k=5} can be found in the Appendix and these graphs also be inspected in the database of interesting graphs at the \textit{House of Graphs}~\cite{hog} by searching for the keywords ``4-critical P5Gem-free'', ``5-critical P5Gem-free'', ``4-critical P5co(P2+P3)-free'', or ``5-critical P5co(P2+P3)-free''.

Note that it follows from~\cite{BHS09,MM12} that there are exactly 12 4-vertex-critical $P_5$-free graphs. As an independent test for the correctness of our implementation, we took these 12 graphs and tested which of these graphs is also gem-free or $\overline{P_3+P_2}$-free. This yielded exactly the same graphs as in \autoref{thm:k=4}. As another correctness test, we modified our program to generate all gem-free graphs and compared it with the known counts of gem-free graphs at the \textit{On-Line Encyclopedia of Integer Sequences}\footnote{\url{https://oeis.org/A079576}}. Also here the results were in complete agreement.

\begin{table}[ht!]
\centering
\small
\begin{tabular}{| l || c | c | c | c | c | c | c | c || c |}
\hline 
Vertices 												& 4		& 5		&	6		&	7	& 	8	&	9	&	10	& 13 	& Total		\\ \hline
4-vertex-critical ($P_5$,gem)-free graphs 				& 1		& 		& 			& 1		& 		& 		& 1		& 		& 3			\\
5-vertex-critical ($P_5$,gem)-free graphs 				& 		& 1		& 			& 		& 		& 3		& 		& 3		& 7			\\
4-vertex-critical ($P_5,\overline{P_3+P_2}$)-free graphs 	& 1		& 		& 1			& 4		& 		& 		& 		& 		& 6			\\
5-vertex-critical ($P_5,\overline{P_3+P_2}$)-free graphs 	& 		& 1		& 			& 1		& 	4	& 14		& 		& 		& 20			\\
\hline 
\end{tabular}
\caption{Counts of all $k$-vertex-critical $(P_5,H)$-free graphs for $k \in \{4,5\}$ and $H \in \{$gem,$ \overline{P_3+P_2}\}$.}

\label{table:counts_vertex_critical_graphs}
\end{table}

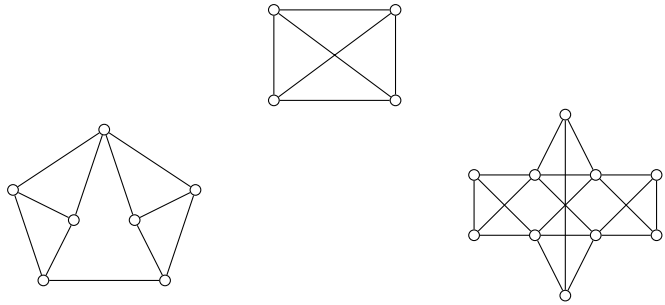
\begin{figure}[tb]
\centering
\begin{subfigure}{.5\textwidth}
\centering
\begin{tikzpicture}[scale=0.4]
\tikzstyle{vertex}=[draw, circle, fill=white!100, minimum width=4pt,inner sep=1pt]
%\draw[step=1cm,color=gray] (-5,-5) grid (5,5);

\node[vertex] (A2) at (2,0) {};
\node[vertex] (A5) at (-2,0) {};
\node[vertex] (A3) at (2,-3) {};
\node[vertex] (A4) at (-2,-3) {};

\draw (A2)--(A3)--(A4)--(A5)--(A2);
\draw (A2)--(A4) (A3)--(A5);

\end{tikzpicture}
\end{subfigure}%

\begin{subfigure}{.5\textwidth}
\centering
\begin{tikzpicture}[scale=0.4]
\tikzstyle{vertex}=[draw, circle, fill=white!100, minimum width=4pt,inner sep=1pt]
%\draw[step=1cm,color=gray] (-5,-5) grid (5,5);

\node[vertex] (A1) at (0,2) {};
\node[vertex] (A2) at (3,0) {};
\node[vertex] (A5) at (-3,0) {};
\node[vertex] (A3) at (2,-3) {};
\node[vertex] (A4) at (-2,-3) {};

\draw (A1)--(A2)--(A3)--(A4)--(A5)--(A1);

\node[vertex] (A6) at (-1,-1) {};
\node[vertex] (A7) at (1,-1) {};
\draw (A6)--(A1) (A6)--(A4) (A6)--(A5) (A7)--(A1) (A7)--(A2) (A7)--(A3);
\end{tikzpicture}
\end{subfigure}%
\begin{subfigure}{.5\textwidth}
\centering
\begin{tikzpicture}[scale=0.4]
\tikzstyle{vertex}=[draw, circle, fill=white!100, minimum width=4pt,inner sep=1pt]
%\draw[step=1cm,color=gray] (-5,-5) grid (5,5);

\node[vertex] (1) at (-2,0) {};
\node[vertex] (2) at (0,0) {};
\node[vertex] (3) at (2,0) {};
\node[vertex] (4) at (4,0) {};

\node[vertex] (5) at (-2,-2) {};
\node[vertex] (6) at (0,-2) {};
\node[vertex] (7) at (2,-2) {};
\node[vertex] (8) at (4,-2) {};

\draw (1)--(2)--(3)--(4) (5)--(6)--(7)--(8);
\draw (1)--(5) (4)--(8);
\draw (1)--(6) (2)--(7) (3)--(8);
\draw (2)--(5) (3)--(6) (4)--(7);

\node[vertex] (9) at (1,-4) {};
\draw (9)--(6) (9)--(7);

\node[vertex] (10) at (1,2) {};
\draw (10)--(2) (10)--(3);
\draw (10)--(9);
\end{tikzpicture}
\end{subfigure}%
\caption{All 3 4-vertex-critical ($P_5$,gem)-free graphs.} %All of these graphs are also 4-critical ($P_5$,gem)-free.}
\label{fig:4vcP5gem}
\end{figure}

\begin{figure}[tb]
\centering
% K5
\begin{subfigure}{.5\textwidth}
\centering
\begin{tikzpicture}[scale=0.4]
\tikzstyle{vertex}=[draw, circle, fill=white!100, minimum width=4pt,inner sep=1pt]
%\draw[step=1cm,color=gray] (-5,-5) grid (5,5);

\node[vertex] (A1) at (0,2) {};
\node[vertex] (A2) at (3,0) {};
\node[vertex] (A5) at (-3,0) {};
\node[vertex] (A3) at (2,-3) {};
\node[vertex] (A4) at (-2,-3) {};

\draw (A1)--(A2)--(A3)--(A4)--(A5)--(A1);
\draw (A1)--(A3)--(A5)--(A2)--(A4)--(A1);

\end{tikzpicture}

\end{subfigure}

% Hajos construction of two copies of K5
\begin{subfigure}{.5\textwidth}
\centering
\begin{tikzpicture}[scale=0.4]
\tikzstyle{vertex}=[draw, circle, fill=white!100, minimum width=4pt,inner sep=1pt]
%\draw[step=1cm,color=gray] (-5,-5) grid (5,5);

\node[vertex] (2) at (0,0) {};
\node[vertex] (1) at (-2,0) {};
\node[vertex] (3) at (0,-2) {};
\node[vertex] (4) at (-2,-2) {};

\draw (1)--(2)--(3)--(4)--(1);
\draw (1)--(3) (2)--(4);

\node[vertex] (6) at (4,0) {};
\node[vertex] (5) at (2,0) {};
\node[vertex] (7) at (4,-2) {};
\node[vertex] (8) at (2,-2) {};

\draw (5)--(6)--(7)--(8)--(5);
\draw (5)--(7) (6)--(8);

\node[vertex] (9) at (1,2) {};

\draw (9)--(1) (9)--(2) (9)--(4) (9)--(5) (9)--(6) (9)--(7);
\draw (3)--(8);

\end{tikzpicture}

\end{subfigure}%
% G_{2,2}
\begin{subfigure}{.5\textwidth}
\centering
\begin{tikzpicture}[scale=0.4]
\tikzstyle{vertex}=[draw, circle, fill=white!100, minimum width=4pt,inner sep=1pt]
%\draw[step=1cm,color=gray] (-5,-5) grid (5,5);

\node[vertex] (1) at (0,0) {};
\node[vertex] (2) at (-1,1) {};
\node[vertex] (3) at (1,1) {};
\node[vertex] (4) at (-3,-1) {};
\node[vertex] (5) at (3,-1) {};

\draw (1)--(2)--(3)--(1);
\draw (4)--(1) (4)--(2) (4)--(3) (5)--(1) (5)--(2) (5)--(3);

\node[vertex] (6) at (-1,-2) {};
\node[vertex] (7) at (1,-2) {};
\node[vertex] (8) at (-1,-4) {};
\node[vertex] (9) at (1,-4) {};

\draw (6)--(7)--(8)--(9)--(6);
\draw (6)--(8) (7)--(9);

\draw (6)--(4)--(8);
\draw (7)--(5)--(9);

\end{tikzpicture}
\end{subfigure}%

\begin{subfigure}{.5\textwidth}
\centering
\begin{tikzpicture}[scale=0.4]
\tikzstyle{vertex}=[draw, circle, fill=white!100, minimum width=4pt,inner sep=1pt]
%\draw[step=1cm,color=gray] (-5,-5) grid (5,5);

\node[vertex] (1) at (-2,0) {};
\node[vertex] (2) at (0,0) {};
\node[vertex] (3) at (2,0) {};
\node[vertex] (4) at (4,0) {};

\node[vertex] (5) at (-2,-2) {};
\node[vertex] (6) at (0,-2) {};
\node[vertex] (7) at (2,-2) {};
\node[vertex] (8) at (4,-2) {};

\draw (1)--(2)--(3)--(4) (5)--(6)--(7)--(8);
\draw (1)--(5) (2)--(6) (3)--(7) (4)--(8);
\draw (1)--(6) (2)--(7) (3)--(8);
\draw (2)--(5) (3)--(6) (4)--(7);

\node[vertex] (9) at (1,-4) {};
\draw (9)--(1) (9)--(5) (9)--(4) (9)--(8);
%\node at (0,-4) {Hajos construction of two copies of $K_5$.};

%\node at (0,-6) {$F$.};

\end{tikzpicture}
%\caption{Hajos construction of two copies of $K_5$.}
\end{subfigure}%
\begin{subfigure}{.5\textwidth}
\centering
\begin{tikzpicture}[scale=0.4]
\tikzstyle{vertex}=[draw, circle, fill=white!100, minimum width=4pt,inner sep=1pt]
%\draw[step=1cm,color=gray] (-5,-5) grid (5,5);

\node[vertex] (1) at (-2,0) {};
\node[vertex] (2) at (0,0) {};
\node[vertex] (3) at (2,0) {};
\node[vertex] (4) at (4,0) {};

\node[vertex] (5) at (-2,-2) {};
\node[vertex] (6) at (0,-2) {};
\node[vertex] (7) at (2,-2) {};
\node[vertex] (8) at (4,-2) {};

\draw (1)--(2)--(3)--(4) (5)--(6)--(7)--(8);
\draw (1)--(5) (2)--(6)  (4)--(8);
\draw (1)--(6) (2)--(7) (3)--(8);
\draw (2)--(5) (3)--(6) (4)--(7);

\node[vertex] (9) at (1,2) {};
\draw (9)--(1) (9)--(5) (9)--(4) (9)--(8) (9)--(7);

\node[vertex] (10) at (-1,-4) {};
\node[vertex] (11) at (1,-5.5) {};
\node[vertex] (12) at (-1,-6) {};
\node[vertex] (13) at (3,-6) {};
\draw (11)--(12)--(13)--(11);
\draw (10)--(11) (10)--(12) (10)--(13) (7)--(11) (7)--(12) (7)--(13);

\draw (10)--(2) (10)--(3) (10)--(6) (10)--(9);

%\node at (0,-9) {$K_5+F$.};

\end{tikzpicture}
%\caption{Hajos construction of two copies of $K_5$.}
\end{subfigure}

% Third row
\begin{subfigure}{.5\textwidth}
\centering
\begin{tikzpicture}[scale=0.4]
\tikzstyle{vertex}=[draw, circle, fill=white!100, minimum width=4pt,inner sep=1pt]
%\draw[step=1cm,color=gray] (-5,-5) grid (5,5);

\node[vertex] (1) at (0,0) {};
\node[vertex] (2) at (-1,1) {};
\node[vertex] (3) at (1,1) {};
\node[vertex] (4) at (-3,-1) {};
\node[vertex] (5) at (3,-1) {};

\draw (1)--(2)--(3)--(1);
\draw (4)--(1) (4)--(2) (4)--(3) (5)--(1) (5)--(2) (5)--(3);

\node[vertex] (6) at (0,-5) {};
\node[vertex] (7) at (-1,-6) {};
\node[vertex] (8) at (1,-6) {};
\node[vertex] (9) at (-3,-4) {};
\node[vertex] (10) at (3,-4) {};

\draw (6)--(7)--(8)--(6);
\draw (9)--(6) (9)--(7) (9)--(8) (10)--(6) (10)--(7) (10)--(8);

\draw (4)--(9) (5)--(10);
\draw (4)--(10) (5)--(9);
\node[vertex] (11) at (-1,-1) {};
\node[vertex] (12) at (-1,-4) {};
\node[vertex] (13) at (1,-2.5) {};
\draw (11)--(12)--(13)--(11);

\draw (4)--(11)--(9) (4)--(12)--(9);
\draw (5)--(13)--(10); 
%\node at (0,-4) {Hajos construction of two copies of $K_5$.};

%\node at (0,-6) {$G_{2,2}$.};

\end{tikzpicture}
%\caption{Hajos construction of two copies of $K_5$.}
\end{subfigure}%
% Fourth row
\begin{subfigure}{.5\textwidth}
\centering
\begin{tikzpicture}[scale=0.4]
\tikzstyle{vertex}=[draw, circle, fill=white!100, minimum width=4pt,inner sep=1pt]
%\draw[step=1cm,color=gray] (-5,-5) grid (5,5);

\node[vertex] (1) at (-1,1) {};
\node[vertex] (6) at (1,1) {};
\node[vertex] (5) at (-3,-1) {};
\node[vertex] (0) at (-1,-1) {};
\node[vertex] (2) at (1,-1) {};
\node[vertex] (12) at (3,-1) {};

\draw (1)--(6)--(5)--(0)--(1);
\draw (1)--(5)(6)--(0);

\draw (1)--(6)--(2)--(12)--(1);
\draw (1)--(2)(6)--(12);

\node[vertex] (9) at (0,-4) {};
\node[vertex] (10) at (-1,-6) {};
\node[vertex] (11) at (1,-6) {};
\node[vertex] (7) at (-3,-4) {};
\node[vertex] (8) at (3,-4) {};

\draw (9)--(10)--(11)--(9);
\draw (7)--(9) (7)--(10) (7)--(11) (8)--(9) (8)--(10) (8)--(11);

\node[vertex] (4) at (-4,-3) {};
\node[vertex] (3) at (4,-3) {};

\draw (7)--(4)--(3)--(8);

\draw (4)--(7)--(5)--(0)--(4);
\draw (4)--(5)(7)--(0);

\draw (3)--(8)--(2)--(12)--(3);
\draw (3)--(2)(8)--(12);

%\node at (0,-4) {Hajos construction of two copies of $K_5$.};

%\node at (0,-6) {$G_{2,2}$.};

\end{tikzpicture}
\end{subfigure}%
\caption{All 7 5-vertex-critical ($P_5$,gem)-free graphs.} %All of these graphs are also 5-critical ($P_5$,gem)-free.}
\label{fig:5vcP5gem}
\end{figure}
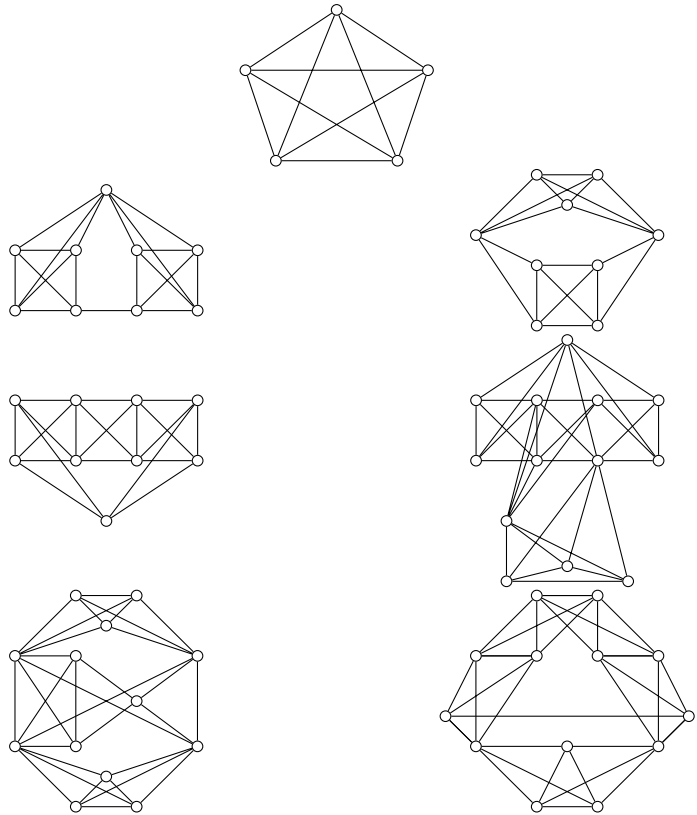

%------------------------------------------------------------------------------------

\section{Conclusion}\label{sec:classification}

In this paper we gave an affirmative answer to the problem posed in~\cite{CGHS21} for $H\in \{$gem$,\overline{P_3+P_2}\}$.
Moreover, we characterized all 4- and 5-vertex-critical ($P_5$,gem)-free graphs
and $(P_5,\overline{P_3+P_2})$-free graphs by extending the computer generation algorithm from~\cite{GS18}.
As future research it is natural to investigate the finiteness of $k$-vertex-critical graphs in the class of $(P_5,H)$-free graphs
for other graphs $H$ of order~5.

%------------------------------------------------------------------------------------

%\bibliographystyle{plain}

%------------------------------------------------------------------------------------

\section*{Appendix}\label{sec:appendix}

Below we list the adjacency lists of all vertex-critical graphs from \autoref{thm:k=4} and \autoref{thm:k=5}. These graphs can also be inspected in the database of interesting graphs at the \textit{House of Graphs}~\cite{hog} by searching for the keywords ``4-critical P5Gem-free'', ``5-critical P5Gem-free'', ``4-critical P5co(P2+P3)-free'', and ``5-critical P5co(P2+P3)-free''.

\subsection*{Adjacency lists of 4-vertex-critical ($P_5$,gem)-free graphs}

\begin{enumerate}
  \item \{0: 1 2 3; 1: 0 2 3; 2: 0 1 3; 3: 0 1 2\}
  \item \{0: 1 4 5; 1: 0 2 5 6; 2: 1 3 6; 3: 2 4 6; 4: 0 3 5; 5: 0 1 4; 6: 1 2 3\}
  \item \{0: 1 4 5; 1: 0 2 5 6 7; 2: 1 3 7; 3: 2 4 6; 4: 0 3 5 6 7; 5: 0 1 4; 6: 1 3 4 8 9; 7: 1 2 4 8 9; 8: 6 7 9; 9: 6 7 8\}
\end{enumerate}

\subsection*{Adjacency lists of 5-vertex-critical ($P_5$,gem)-free graphs}

\begin{enumerate}
  \item \{0: 1 2 3 4; 1: 0 2 3 4; 2: 0 1 3 4; 3: 0 1 2 4; 4: 0 1 2 3\}
  \item \{0: 1 4 5 6; 1: 0 2 5 6 7 8; 2: 1 3 7 8; 3: 2 4 7 8; 4: 0 3 5 6; 5: 0 1 4 6; 6: 0 1 4 5; 7: 1 2 3 8; 8: 1 2 3 7\}
  \item \{0: 1 4 5 6; 1: 0 2 5 6; 2: 1 3 6 7 8; 3: 2 4 7 8; 4: 0 3 5 7 8; 5: 0 1 4 6; 6: 0 1 2 5; 7: 2 3 4 8; 8: 2 3 4 7\}
  \item \{0: 1 4 5 6 7; 1: 0 2 5 6 8; 2: 1 3 6 8; 3: 2 4 7 8; 4: 0 3 5 7; 5: 0 1 4 6 7; 6: 0 1 2 5 8; 7: 0 3 4 5; 8: 1 2 3 6\}
  \item \{0: 1 4 5 6; 1: 0 2 5 6 7 8 12; 2: 1 3 7 12; 3: 2 4 6 7 8 12; 4: 0 3 5 8; 5: 0 1 4 6; 6: 0 1 3 5 9 10 11; 7: 1 2 3 12; 8: 1 3 4 9 10 11; 9: 6 8 10 11; 10: 6 8 9 11; 11: 6 8 9 10; 12: 1 2 3 7\}
  \item \{0: 1 4 5 6 7 8; 1: 0 2 5 6; 2: 1 3 6 7 8 12; 3: 2 4 8 12; 4: 0 3 5 7 12; 5: 0 1 4 6 7 8; 6: 0 1 2 5; 7: 0 2 4 5 9 10 11; 8: 0 2 3 5 9 10 11 12; 9: 7 8 10 11; 10: 7 8 9 11; 11: 7 8 9 10; 12: 2 3 4 8\}
  \item \{0: 1 4 5 6 7 8; 1: 0 2 5 6 12; 2: 1 3 6 7 8 12; 3: 2 4 8 12; 4: 0 3 5 7; 5: 0 1 4 6 7 8; 6: 0 1 2 5 12; 7: 0 2 4 5 9 10 11 12; 8: 0 2 3 5 9 10 11 12; 9: 7 8 10 11; 10: 7 8 9 11; 11: 7 8 9 10; 12: 1 2 3 6 7 8\}
\end{enumerate}

\subsection*{Adjacency lists of 4-vertex-critical $(P_5,\overline{P_3+P_2})$-free graphs}

\begin{enumerate}
  \item \{0: 1 2 3; 1: 0 2 3; 2: 0 1 3; 3: 0 1 2\}
  \item \{0: 1 4 5; 1: 0 2 5; 2: 1 3 5; 3: 2 4 5; 4: 0 3 5; 5: 0 1 2 3 4\}
  \item \{0: 1 4 5; 1: 0 2 5 6; 2: 1 3 6; 3: 2 4 6; 4: 0 3 5; 5: 0 1 4; 6: 1 2 3\}
  \item \{0: 1 4 5 6; 1: 0 2 5 6; 2: 1 3 5; 3: 2 4 6; 4: 0 3 5; 5: 0 1 2 4; 6: 0 1 3\}
  \item \{0: 1 4 5 6; 1: 0 2 5 6; 2: 1 3 5 6; 3: 2 4 6; 4: 0 3 5; 5: 0 1 2 4; 6: 0 1 2 3\}
  \item \{0: 1 4 5 6; 1: 0 2 5 6; 2: 1 3 6; 3: 2 4 6; 4: 0 3 5; 5: 0 1 4; 6: 0 1 2 3\}
\end{enumerate}

\subsection*{Adjacency lists of 5-vertex-critical $(P_5,\overline{P_3+P_2})$-free graphs}

\begin{enumerate}
  \item \{0: 1 2 3 4; 1: 0 2 3 4; 2: 0 1 3 4; 3: 0 1 2 4; 4: 0 1 2 3\}
  \item \{0: 1 4 5 6; 1: 0 2 5 6; 2: 1 3 5 6; 3: 2 4 5 6; 4: 0 3 5 6; 5: 0 1 2 3 4 6; 6: 0 1 2 3 4 5\}
  \item \{0: 1 4 5 6 7; 1: 0 2 5 6 7; 2: 1 3 5 7; 3: 2 4 6 7; 4: 0 3 5 7; 5: 0 1 2 4 7; 6: 0 1 3 7; 7: 0 1 2 3 4 5 6\}
  \item \{0: 1 4 5 6 7; 1: 0 2 5 6 7; 2: 1 3 6 7; 3: 2 4 6 7; 4: 0 3 5 6; 5: 0 1 4 6; 6: 0 1 2 3 4 5 7; 7: 0 1 2 3 6\}
  \item \{0: 1 4 5 6 7; 1: 0 2 5 6 7; 2: 1 3 5 6 7; 3: 2 4 6 7; 4: 0 3 5 7; 5: 0 1 2 4 7; 6: 0 1 2 3 7; 7: 0 1 2 3 4 5 6\}
  \item \{0: 1 4 5 6; 1: 0 2 5 6 7; 2: 1 3 6 7; 3: 2 4 6 7; 4: 0 3 5 6; 5: 0 1 4 6; 6: 0 1 2 3 4 5 7; 7: 1 2 3 6\}
  \item \{0: 1 4 5 6; 1: 0 2 5 6 7 8; 2: 1 3 7 8; 3: 2 4 7 8; 4: 0 3 5 6; 5: 0 1 4 6; 6: 0 1 4 5; 7: 1 2 3 8; 8: 1 2 3 7\}
  \item \{0: 1 4 5 6 7 8; 1: 0 2 5 6 7 8; 2: 1 3 6 7; 3: 2 4 7 8; 4: 0 3 5 6 8; 5: 0 1 4 6 8; 6: 0 1 2 4 5; 7: 0 1 2 3 8; 8: 0 1 3 4 5 7\}
  \item \{0: 1 4 5 6 7 8; 1: 0 2 5 6 7; 2: 1 3 5 6 8; 3: 2 4 6 7 8; 4: 0 3 5 7 8; 5: 0 1 2 4 8; 6: 0 1 2 3 7 8; 7: 0 1 3 4 6; 8: 0 2 3 4 5 6\}
  \item \{0: 1 4 5 6; 1: 0 2 5 6; 2: 1 3 6 7 8; 3: 2 4 7 8; 4: 0 3 5 7 8; 5: 0 1 4 6; 6: 0 1 2 5; 7: 2 3 4 8; 8: 2 3 4 7\}
  \item \{0: 1 4 5 6 7; 1: 0 2 5 6 8; 2: 1 3 6 8; 3: 2 4 7 8; 4: 0 3 5 7; 5: 0 1 4 6 7; 6: 0 1 2 5 8; 7: 0 3 4 5; 8: 1 2 3 6\}
  \item \{0: 1 4 5 6 7; 1: 0 2 5 6 7; 2: 1 3 6 7 8; 3: 2 4 7 8; 4: 0 3 5 7 8; 5: 0 1 4 6 8; 6: 0 1 2 5 8; 7: 0 1 2 3 4; 8: 2 3 4 5 6\}
  \item \{0: 1 4 5 6 7; 1: 0 2 5 6 7; 2: 1 3 6 7 8; 3: 2 4 7 8; 4: 0 3 5 7 8; 5: 0 1 4 6; 6: 0 1 2 5; 7: 0 1 2 3 4 8; 8: 2 3 4 7\}
  \item \{0: 1 4 5 6 7 8; 1: 0 2 5 6 7; 2: 1 3 6 7; 3: 2 4 7 8; 4: 0 3 5 7 8; 5: 0 1 4 6 8; 6: 0 1 2 5 8; 7: 0 1 2 3 4; 8: 0 3 4 5 6\}
  \item \{0: 1 4 5 6 7 8; 1: 0 2 5 6 7; 2: 1 3 6 7 8; 3: 2 4 7 8; 4: 0 3 5 7 8; 5: 0 1 4 6; 6: 0 1 2 5 8; 7: 0 1 2 3 4 8; 8: 0 2 3 4 6 7\}
  \item \{0: 1 4 5 6 7 8; 1: 0 2 5 6 7 8; 2: 1 3 6 7 8; 3: 2 4 7 8; 4: 0 3 5 7; 5: 0 1 4 6; 6: 0 1 2 5 8; 7: 0 1 2 3 4 8; 8: 0 1 2 3 6 7\}
  \item \{0: 1 4 5 6 7; 1: 0 2 5 6 7 8; 2: 1 3 6 7 8; 3: 2 4 7 8; 4: 0 3 5 6; 5: 0 1 4 6; 6: 0 1 2 4 5; 7: 0 1 2 3 8; 8: 1 2 3 7\}
  \item \{0: 1 4 5 6 7; 1: 0 2 5 6 7 8; 2: 1 3 6 7 8; 3: 2 4 7 8; 4: 0 3 5 6; 5: 0 1 4 6 8; 6: 0 1 2 4 5; 7: 0 1 2 3 8; 8: 1 2 3 5 7\}
  \item \{0: 1 4 5 6 8; 1: 0 2 5 6 7 8; 2: 1 3 6 7; 3: 2 4 6 7; 4: 0 3 5 6 8; 5: 0 1 4 6 8; 6: 0 1 2 3 4 5; 7: 1 2 3 8; 8: 0 1 4 5 7\}
  \item \{0: 1 4 5 6; 1: 0 2 5 6 7 8; 2: 1 3 6 7 8; 3: 2 4 7 8; 4: 0 3 5 6; 5: 0 1 4 6; 6: 0 1 2 4 5; 7: 1 2 3 8; 8: 1 2 3 7\}
\end{enumerate}

\end{document}